\documentclass[11pt,reqno]{amsart}
\usepackage[utf8]{inputenc}

\usepackage{amsfonts}
\usepackage{amsmath}
\usepackage{amssymb}
\usepackage{amsthm}
\usepackage{array}
\usepackage{calc}
\usepackage{caption}
\usepackage{color}
\usepackage{dashbox}
\usepackage{dutchcal}
\usepackage{enumitem}
\usepackage{extarrows}
\usepackage[hang,flushmargin,symbol*]{footmisc}
\usepackage[margin=1in]{geometry}
\usepackage{graphicx}
\usepackage[colorlinks=true, citecolor=black, urlcolor=black, linkcolor=black]{hyperref}
\usepackage[capitalise]{cleveref}
\usepackage{mathdots}
\usepackage{mathrsfs}
\usepackage{mathtools}
\usepackage{setspace}
\usepackage{stmaryrd}
\usepackage{tikz-cd}
\usetikzlibrary{arrows,shapes,positioning}
\usetikzlibrary{decorations.markings}
\usetikzlibrary{patterns}
\usepackage[normalem]{ulem}
\usepackage{url}
\usepackage{verbatim}
\usepackage[colorinlistoftodos,loadshadowlibrary]{todonotes} 

\DeclareUnicodeCharacter{221E}{\ensuremath{\infty}}

\newtheorem{theorem}{Theorem}[section]
\newtheorem{lemma}[theorem]{Lemma}

\newtheorem{corollary}[theorem]{Corollary}
\newtheorem{proposition}[theorem]{Proposition}

\theoremstyle{definition}
\newtheorem{definition}[theorem]{Definition}
\newtheorem{example}[theorem]{Example}
\newtheorem{remark}[theorem]{Remark}

\newcommand{\scr}[1]{\ensuremath{\mathscr{#1}}}

\renewcommand{\int}{{\rm int}}

\newcommand{\ZZ}{\mathbb{Z}}
\newcommand{\CC}{\mathbb{C}}
\newcommand{\NN}{{\mathbb{N}}}
\newcommand{\QQ}{{\mathbb{Q}}}
\newcommand{\OO}{\mathcal{O}}

\newcommand{\RR}{\mathbb{R}}
\newcommand{\Aff}{{\mathbb{A}}}
\newcommand{\PP}{\mathbb{P}}
\newcommand{\GG}{\mathbb{G}}

\newcommand{\Spec}{{\rm{Spec}\:}}

\newcommand{\lkah}[1]{\Omega^{\rm log}_{#1}}
\newcommand{\kah}[1]{\Omega_{#1}}

\newcommand{\lpbstrict}{\pb}
\newcommand{\pb}{{\arrow[dr, phantom, very near start, "\ulcorner"]}}

\def\overnorm#1{\overline{#1}\vphantom{#1}}
\renewcommand{\bar}[1]{\ensuremath{\overnorm{#1}}}

\newcommand{\bra}[1]{[{#1}]}

\newcommand{\LL}{\mathbb{L}}

\newcommand{\pt}{{\rm pt}}

\newcommand{\bb}[1]{\ensuremath{\mathbb{#1}}}

\newcommand{\var}[1]{{\rm Var}_{#1}}
\newcommand{\logvar}[1]{{\rm LogSch}_{#1}}
\newcommand{\klvar}[1]{K_0({\rm LogSch}_{#1})}

\renewcommand{\int}[1]{{\rm int}\,#1}

\newcommand{\ordkvar}[1]{K_0({\rm Var}_{#1})}

\newcommand{\KN}[1]{#1^{KN}}

\newcommand{\Elog}{{E^{\rm log}}}
\newcommand{\Elogbar}{\overline{E}^{\rm log}}

\renewcommand{\t}{t}
\newcommand{\tbar}{\overline{t}}

\DeclareMathOperator{\MHS}{MHS}

\newcommand{\Bitt}{{\rm Bitt}}

\renewcommand{\tilde}[1]{\widetilde{#1}}
\renewcommand{\hat}[1]{\widehat{#1}}

\newcommand{\lcf}{\mathsf{lcf}}

\usepackage[foot]{amsaddr}
\title{The log Grothendieck ring of varieties}
\author{Andreas Gross}
\author{Leo Herr}
\author{David Holmes}
\author{Pim Spelier}
\author{Jesse Vogel}

\address{Addresses: Goethe University Frankfurt, Virginia Tech, Leiden University, Utrecht University, formerly Leiden University}
\email{herr@vt.edu}

\date{\today}

\begin{document}

\maketitle

\begin{abstract}
    We define a Grothendieck ring of varieties for log schemes. It is generated by one additional class ``$P$'' over the usual Grothendieck ring. 
    
    We show the na\"ive definition of log Hodge numbers does not make sense for all log schemes. We offer an alternative that does. 
\end{abstract}

\section{Introduction}

This article
\begin{itemize}
    \item Defines a Grothendieck group of varieties $\klvar{k}$ for log schemes over a base field $k$. 
    \item Provides a presentation of $\klvar{k}$ in terms of the usual Grothendieck group of varieties with one generator $P$ and one relation. 
    \item Deduces consequences for log Hodge numbers. 
\end{itemize}

\subsubsection*{Overview}

The mixed Hodge structure $\MHS(X)$ of a scheme $X$ over $\CC$ is an invariant that generalises many classical invariants such as Euler characteristics, Betti numbers, and Hodge numbers. This is a \emph{motivic invariant}, meaning that a nice closed embedding $Z \to X$ gives rise to an exact triangle \[\MHS(Z) \to \MHS(X) \to \MHS(X \setminus Z) \xrightarrow{+1}\] in the derived category of mixed Hodge structures. 
On the level of classical invariants such as the Euler characteristic, this simply means $\chi(X) = \chi(Z) + \chi(X \setminus Z)$.

When working with a family of schemes $X \to S$, one obtains a \emph{family} of mixed Hodge structures $\MHS(X/S)$ (also called a variation) over $S$ \cite{Griffiths1}. The moduli space of all (polarised) Hodge structures was first constructed by Griffiths, and forms an analytic, highly non-compact space. In \cite{Griffiths5} he presented the dream of (partially) compactifying this period domain. The answer was given by \cite{loghodgekatousui}, where a full compactification is constructed using \emph{logarithmic} Hodge structures.

Unfortunately, the question of how to obtain families of logarithmic Hodge structures from families of schemes is still open even in simple situations \cite[Remark 2.6.1]{loghodgeoverstandardlogpoint} such as non-compact varieties.

As a precursor to a future logarithmic mixed Hodge theory,
we study \emph{logarithmic Hodge numbers} $\dim_\CC H^q(\wedge^p \lkah{X})$, where $\lkah{X}$ is the sheaf of log differentials \cite[IV.1]{ogusloggeom}. Classically, the Hodge numbers of a smooth projective variety $X$ are defined by $h^{p, q}(X) \coloneqq \dim_\CC H^q(\wedge^p \kah{X})$. These are invariants of the Hodge structure of $X$. As the Hodge structure is a motivic invariant of $X$, the Hodge numbers only depend on the class $[X]$ in the \emph{Grothendieck ring} $\ordkvar{\CC}$, the free abelian group on varieties modulo the scissors relations $[X] = [Z] + [X \setminus Z]$ for a closed embedding $Z \to X$.

This in particular helps compute Hodge numbers $h^{p,q}(Y)$ for non-projective or singular varieties $Y$.

This paper defines the \emph{logarithmic Grothendieck ring} $\klvar{\CC}$. In Section~\ref{sec:presentation} we compute a presentation for this ring in terms of the classical ring $\ordkvar{\CC}$. This allows the computation of logarithmic motivic invariants as in the classical case.

In Section~\ref{sec:hodge} we show that the logarithmic Hodge numbers \emph{do not} give motivic invariants. We prove that the Euler characteristics $\chi(\wedge^p \lkah{X})$ are in fact motivic invariants.

\subsection{The log Grothendieck ring}

Write $\var{k}$ for the category of finite type, separated $k$-schemes and $\logvar{k}$ for the category of fine and saturated (f.s.) $k$-log schemes with underlying scheme in $\var{k}$.

\begin{definition}
The log Grothendieck ring
\[
\klvar{k}
\]
is the free abelian group generated by isomorphism classes $[X]$ for $X \in \logvar{k}$, modulo the \emph{strict scissor relations}
    \begin{equation}\label{intro:eqn:strictscissor}
    [X] = [Z] + [X \setminus Z]
    \end{equation}
    for all strict closed immersions $Z \to X$,
    and the \emph{log blowup relations}
    \begin{equation}\label{intro:eqn:logmodreln}
    [\tilde{X}] = [X]
    \end{equation}
    for all log blowups $\tilde{X} \to X$.
    The ring structure is induced by the product on $\logvar{k}$.

\end{definition}

Take $k = \CC$. We want to define and compute log Hodge numbers using $\klvar{\CC}$.

If $X$ is a smooth, projective variety, its Hodge numbers $h^{p, q}(X)$ are 
\[h^{p, q}(X) \coloneqq \dim_\CC H^q(\wedge^p \kah{X}) \qquad \in \NN.\]
The $e$-polynomial of $X$ is the generating function\footnote{Warning: there are two common sign conventions found in literature, namely $e(u,v)$ and $e(-u,-v)$.} 
\[e(X) \coloneqq \sum h^{p, q}(X) u^p v^q \qquad \in \ZZ[u, v].\]
Both factor through the Grothendieck ring
\begin{equation}\label{eqn:hodgeepolygrothringmap}
    h^{p, q} : \ordkvar{\CC} \to \NN, \qquad e : \ordkvar{\CC} \to \ZZ[u, v],
\end{equation}
and $e$ is even a ring homomorphism. For schemes $X$ which are not projective or smooth, Hodge numbers may be defined by the maps \eqref{eqn:hodgeepolygrothringmap} or using mixed Hodge structures.

For a log smooth, projective $X$ over $\CC$, define its log Hodge numbers 
\[
h^{p, q}_{\rm log}(X) \coloneqq \dim_\CC H^q(\wedge^p \lkah{X}) \qquad \in \NN
\]
and log $e$-polynomial
\[\Elog(X) \coloneqq \sum h^{p, q}_{\rm log}(X) u^p v^q \qquad \in \ZZ[u, v].\]

We want to define these log Hodge numbers and $e$-polynomials, in a motivic way, for log schemes which are not projective or log smooth. This is impossible. 

\begin{proposition}[Proposition~\ref{prop:nologhodge}]
\label{intro:prop:nologhodge}
    There is no map
    \[
    \phi: \logvar{\CC} \to \ZZ[u,v]
    \]
    that satisfies the strict scissor equation $\phi(X) = \phi(Z) + \phi(X \setminus Z)$ and that agrees with $\Elog$ on smooth, log smooth, projective log schemes.
\end{proposition}

Example~\ref{ex:Epolynotwelldef} is an explicit counterexample. 

How can we fix this problem? Write 
\[\Elogbar_1(X) \coloneqq \Elog(X)(u, -1)\]
for the image of $\Elog(X)$ in the polynomial ring
\[\ZZ[u, v]/(v + 1) = \ZZ[u].\]
This is the generating function of the Euler characteristics of the wedge powers $\wedge^p \lkah{X}$
\[\Elogbar_1(X) = \sum \chi(\wedge^p \lkah{X}) u^p.\]

\begin{theorem}[Theorems~\ref{thm:E=t}, \ref{thm:E=tlogsmoothandreasgross}]
    There exists a ring homomorphism $\bar t_1 : \klvar{\CC} \to \ZZ[u]$ that satisfies $\bar t_1([X])=\Elogbar_1(X)$ for all $X$ with smooth, projective underlying scheme that are either
    \begin{enumerate}
        \item\label{eit:logsmooth} log smooth log scheme, or 
        \item\label{eit:constfree} constant free log schemes (see Definition \ref{def:constant log structure}). 
    \end{enumerate}
    
    In particular, the polynomial $\Elogbar_1$ is invariant under sufficiently fine log modifications $\tilde X \to X$ where $X$ belongs to \eqref{eit:logsmooth} or \eqref{eit:constfree} and $X, \tilde X$ have smooth, projective underlying scheme: 
    \[\Elogbar_1(\tilde X) = \Elogbar_1(X).\] 
\end{theorem}

The log Hodge numbers $h^{p, q}_{\rm log}$ are not well-defined for log schemes which are not log smooth or projective, but their alternating sums $\sum (-1)^q h^{p, q}_{\mathrm log}$, which should be thought of as the holomorphic Euler characteristics $\chi(\wedge^p \lkah{X})$, are.  

We construct $\bar t_1$ using a presentation of $\klvar{\CC}$ as a $\ordkvar{\CC}$-algebra. Write $P \coloneqq (\Spec \CC, \CC^*\oplus \NN)$ for the \emph{standard log point} with rank-one log structure, and also $P \coloneqq [P] \in \klvar{\CC}$ for its class. 

\begin{theorem}[{Theorem \ref{thm:vogelpresentation}}]\label{inthm:vogelpresentation}
    The Grothendieck ring of log varieties is generated over the ordinary Grothendick ring by $P$, subject to a single relation
    \begin{equation}\label{eqn:inthmvogelpresentation}
    \klvar{\CC} \simeq \dfrac{\ordkvar{\CC}[P]}{P^2 + P [\GG_m]}.
    \end{equation}
\end{theorem}

The relation $P^2 + P[\GG_m] = 0$ is derived in Example~\ref{ex:vogelrlns}. Define two $\ordkvar{\CC}$-algebra homomorphisms
\[\tau, \rho : \klvar{\CC} \to \ordkvar{\CC}\]
by 
\[\tau(P) = 0, \qquad \rho(P) = -[\GG_m].\]
Then the map $\bar t_1$ from Theorem \ref{thm:E=t} is precisely the composite
\[\klvar{\CC} \overset{\rho}{\longrightarrow} \ordkvar{\CC} \overset{e}{\longrightarrow} \ZZ[u, v] \overset{v = -1}{\longrightarrow} \ZZ[u].\]

We construct another invariant of log schemes using the presentation \eqref{eqn:inthmvogelpresentation}. Write $\chi_c$ for the (compactly supported) Euler characteristic 
\[\chi_c(X) \coloneqq \sum (-1)^q \dim_\QQ H^q_c (X, \QQ)\]
and for its extension $\chi_c : \ordkvar{\CC} \to \ZZ$ to the Grothendieck group.

\begin{proposition}[{Proposition \ref{prop:logeulerchar}}]
    The composite 
    \[\chi^{\rm log} : \klvar{\CC} \overset{\tau}{\longrightarrow} \ordkvar{\CC} \overset{\chi}{\longrightarrow} \ZZ\]
    is the unique extension of the ring homomorphism $\chi_c$ to $\klvar{\CC}$. 
\end{proposition}

We compute the class $[X] \in \klvar{k}$ of any toric variety $X$, which gives all the above invariants. 

\begin{proposition}[{Proposition \ref{prop:toricclass}}]
    Let $\Sigma \subseteq \RR^n$ be a fan and $X$ the associated $n$-dimensional toric variety. 
    Its class in $\klvar{k}$ is 
    \[
    \bra{X} = \bra{\GG_m}^n + (1-\chi_c(\Sigma)) \cdot P \bra{\GG_m}^{n-1}, 
    \]
    where $\chi_c(\cdot)$ is the Euler characteristic with compact support.
\end{proposition}

The $\Elogbar_1$-polynomial and log Euler characteristic $\chi^{\rm log}$ of such a toric variety $X$ are then 
\[
\Elogbar_1(X) = \chi_c(\Sigma) \cdot e(\GG_m)^n = \chi_c(\Sigma) \cdot (-u - 1)^n
\]
\[
\chi^{\rm log}(X) = \chi_c(\GG_m)^n = 0.
\]

\subsubsection*{Acknowledgments}

We are grateful to
Yagna Dutta,
Tommaso de Fernex,
Márton Hablicsek,
Karl Schwede, 
Y.P. Lee, and
Alexander Zotine
for conversations on the present paper. 
Appendix \ref{appendix:rothemail} is due to Mike Roth, who generously allowed the authors to include his arguments.

During the writing of this article, the NSF RTG grant \#1840190 supported L.H., the NWO grant VI.Vidi.193.006 supported L.H., D.H., and P.S., and the ERC Consolidator Grant FourSurf 101087365 supported P.S. A.G. has received funding from the Deutsche Forschungsgemeinschaft (DFG, German Research Foundation) TRR 326 \emph{Geometry and Arithmetic of Uniformized Structures}, project number 444845124, and from the Marie-Sk\l{}odowska-Curie-Stipendium Hessen (as part of the HESSEN HORIZON initiative).

\subsection{Conventions}

All our log schemes and log algebraic stacks are f.s.~ (fine and saturated, \cite[Chapter 2.1]{ogusloggeom}).

We work over a field $k$ of any characteristic, not necessarily algebraically closed, except for Section \ref{sec:hodge} where we take $k = \CC$.

We write $\logvar{k}$ for the category of finite type, separated log schemes over $\Spec k$, where $\Spec k$ has the trivial log structure. All fiber squares in this paper are both log fiber squares and scheme-theoretic fiber squares.

An \emph{s.n.c.~pair} is a pair $(X, D)$ of a \emph{smooth} scheme $X$ and a strict normal crossings divisor $D$, with its natural divisorial log structure. We endow a toric variety $T$ like $\Aff^n, \PP^n$ with its natural toric log structure by default, which comes from the toric divisors.

A \emph{log blowup} $\tilde{X} \to X$ of a log scheme $X$ in a log ideal $I \subseteq M_X$ is the universal log scheme over $X$ such that the pullback of $I$ is invertible, see \cite[Section~III.2.6]{ogusloggeom}. 
For an s.n.c.~pair $(X,D)$, this is an iterated blowup $\tilde{X} \to X$ in boundary strata. For a toric variety, this is a toric blowup, which can be described by subdivision of the fan. Log blowups are always proper, log \'etale log monomorphisms.

A \emph{log modification} $\tilde{X} \to X$ is a map such that there is some log blowup $\hat{X} \to \tilde{X}$ such that the composite $\hat{X} \to \tilde X \to X$ is also a log blowup.

\section{The log Grothendieck ring and its presentation}
\label{sec:presentation}
In this section we will prove the following simple presentation of the log Grothendieck ring.

\begin{theorem}\label{thm:vogelpresentation}

The log Grothendieck ring is generated by $P$ with one relation: 
\begin{equation}\label{eqn:vogelpresentation}
\klvar{k} = \dfrac{\ordkvar{k}[P]}{P^2 + P [\GG_{m}]}. 
\end{equation}
\end{theorem}

We first show the relation $P^2 + P [\GG_{m}]$ holds.

\begin{example}
\label{ex:vogelrlns}
We take $X$ to be $\bb A^2$ with its toric log structure given by the toric boundary, and $\tilde{X} \to X$ the blowup in $(0,0)$. If we write $P$ for the class in $\klvar{k}$ of the origin with point, then we find $[X] = [\bb G_m]^2 + 2 [\bb G_m] \cdot P + P^2$, and $[\tilde{X}] = [\bb G_m]^2 + 3 [\bb G_m] \cdot P + 2 P^2$, as illustrated in Figure~\ref{fig:vogelrelns}. The equality $[X] = [\tilde{X}]$ then gives rise to the relation \begin{equation}
    P([\bb G_m] + P) = 0
\end{equation}
in $\klvar{k}$.

\end{example}

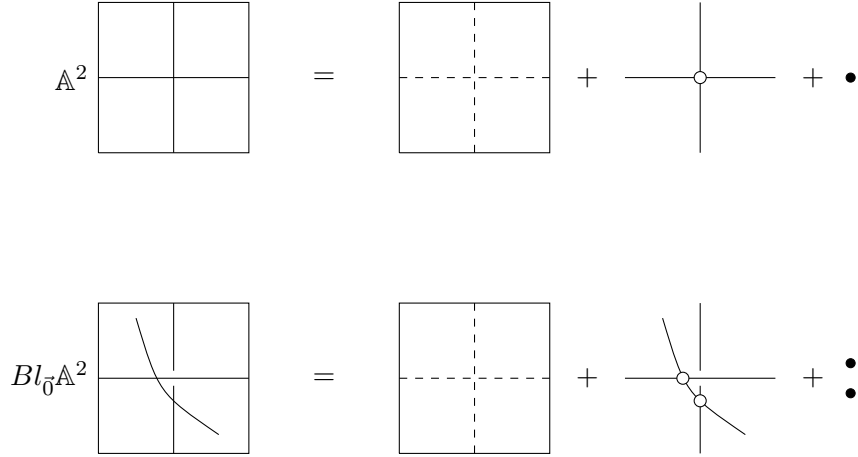
\begin{figure}
    \centering
    \begin{tikzpicture}
    \begin{scope}[shift = {(0,4)}]
    
    \node[left] at (0, 1){$\Aff^2$};
    
    \node at (3, 1){$=$};
    \node at (6.5, 1){$+$};
    \node at (9.5, 1){$+$};

    \begin{scope}
        \draw (0, 0) to (2, 0) to (2, 2) to (0, 2) to (0, 0);
        \draw (1, 0) to (1, 2);
        \draw (0, 1) to (2, 1);
    \end{scope}
    \begin{scope}[shift = {(4, 0)}]
        \draw (0, 0) to (2, 0) to (2, 2) to (0, 2) to (0, 0);
        \draw[dashed] (1, 0) to (1, 2);
        \draw[dashed] (0, 1) to (2, 1);
    \end{scope}
    \begin{scope}[shift = {(7, 0)}]
        \draw (1, 0) to (1, 2);
        \draw (0, 1) to (2, 1);
        \filldraw[fill=white, draw=black] (1, 1) circle (.08);
    \end{scope}
    \begin{scope}[shift = {(9, 0)}]
        \filldraw[black] (1, 1) circle (.06);
    \end{scope}
    \end{scope}
    \begin{scope}[shift = {(0, 0)}]
        
        \node[left] at (0, 1){$Bl_{\vec 0}\Aff^2$};
        
        \node at (3, 1){$=$};
        \node at (6.5, 1){$+$};
        \node at (9.5, 1){$+$};
        
        \draw (1, 0) to (1, 2);
        \draw[line width = 2mm, color=white] (0, 1) to (2, 1);
        \draw (0, 1) to (2, 1);
        \draw (0, 0) to (2, 0) to (2, 2) to (0, 2) to (0, 0);
        \draw (.5, 1.8) .. controls (.8, .8) .. (1.6, .25);

        \begin{scope}[shift = {(4, 0)}]
        \draw (0, 0) to (2, 0) to (2, 2) to (0, 2) to (0, 0);
        \draw[dashed] (1, 0) to (1, 2);
        \draw[dashed] (0, 1) to (2, 1);
        \end{scope}
        
        \begin{scope}[shift = {(7, 0)}]\draw (1, 0) to (1, 2);
        \draw[line width = 2mm, color=white] (0, 1) to (2, 1);
        \draw (0, 1) to (2, 1);
        \draw (.5, 1.8) .. controls (.8, .8) .. (1.6, .25);
        \filldraw[fill=white, draw=black] (.77, 1) circle (.08);
        \filldraw[fill=white, draw=black] (1, .7) circle (.08);
        \end{scope}
        \begin{scope}[shift = {(9, 0)}]
            \filldraw[black] (1, 1.2) circle (.06);
            \filldraw[black] (1, .8) circle (.06);
        \end{scope}
    \end{scope}
    \end{tikzpicture}
    \caption{Decompose the plane $\Aff^2$ and its blowup $Bl_{\vec 0} \Aff^2$ at the origin into locally closed strata on which the log structure is constant. Obtain $[\Aff^2] = [\GG_m]^2 + 2 [\GG_m] \cdot P + P^2$ and $[Bl_{\vec 0} \Aff^2] = [\GG_m]^2 + 3 [\GG_m] \cdot P + 2 P^2$. }
    \label{fig:vogelrelns}
\end{figure}

Next we will prove $\klvar{k}$ is generated by log schemes with constant free log structure.

\begin{definition}
\label{def:constant log structure}
A log scheme $X$ is \emph{constant} if its log structure decomposes as a direct sum
\[M_X = \OO_X^* \oplus \underline{Q}\]
for a sharp f.s.\ monoid $Q$. Write $(X, Q)$ for the resulting constant log scheme.

A log scheme $X$ is \emph{constant free} if it is constant $X = (X, Q)$ and $Q \simeq \NN^r$ is free.  
Say $X$ is \emph{locally constant free} if it admits a locally closed stratification $X = \bigsqcup X_\sigma$ with $X_\sigma$ constant free. 

We write $\logvar{k}^\mathsf{c}, \logvar{k}^{\mathsf{cf}}$ and $\logvar{k}^\lcf$ for the full subcategories of $\logvar{k}$ consisting of constant, constant free, and locally constant free log schemes.
\end{definition}

Constant log schemes are exactly the log schemes that admit strict maps to constant points $(\pt,Q)$. To prove that the log Grothendieck ring is generated by $P$ over $\ordkvar{k}$, we first need the following lemma.

\begin{lemma}\label{lem:noethdecomposelinebundles}
    Let $(L_i)$ be a finite set of line bundles on a noetherian scheme $X$. There is a locally closed stratification $X = \bigsqcup X_i$ on which all the $L_i$'s are trivial. 
\end{lemma}

\begin{corollary}
\label{cor:locconstequivconst}
    If $X$ is a locally constant log scheme with stalks $\bar M_{X, \bar x} \simeq N$, write $X' = (X, N)$ for the same scheme with constant log structure. Then $X$ admits a locally closed stratification $X = \bigsqcup X_i$ with $X_i \cong (X_i, N)$, and we have an equality of classes $[X] = [X']$ in $\klvar{k}$.
\end{corollary}

This says $[X]$ only depends on $X^\circ$ and the sheaf $\bar M_X$.

\begin{proposition}
\label{prop:niceblowup}
A log scheme $X \in \logvar{k}$ admits a locally constant free log blowup $\tilde{X} \in \logvar{k}^\lcf$.
\end{proposition}

\begin{proof}
    
    Take a log blowup of $X$ that admits charts Zariski locally \cite[Theorem 5.4]{niziol}. Refine this log blowup by a log blowup $\tilde{X}$ on which the cones are free as in the proof of \cite[Theorem 11]{originaltoricvarsbook}, so the stalks $\bar M_{\tilde{X},x} \simeq \NN^r$ are free monoids. There is a locally closed stratification on which its characteristic monoid is constant. We are done by Corollary~\ref{cor:locconstequivconst}.

\end{proof}

\begin{corollary}
\label{cor:psisurjective}
The map 
\begin{equation}
\label{eqn:vogelpresmap}
    \Psi: \ordkvar{k}[P] \to \klvar{k}
\end{equation}
sending $[X] P^k$ to the constant log scheme $(X, \NN^k) = X \times P^k$ is surjective.
\end{corollary}

It remains to show that $\ker \Psi$ is generated by $P(P+\bra{\GG_m})$. For this, we need two final propositions.

\begin{definition}
    Let $X$ be a locally constant free log scheme. Take the locally closed stratification $X = \bigsqcup_{i \geq 0} X_i$ where $X_i$ is the locus where the rank of $\bar M_X$ is $i$. We define \[[X]_{\lcf} = \sum_i [X_i] \cdot P^i \in \ordkvar{k}[P].\]
\end{definition}

\begin{proposition}
    \label{prop:basicblowuprelns}
    Let $I \subset \ordkvar{k}[P]$ be the ideal generated by elements
    \[ [X]_\lcf  - [\tilde{X}]_\lcf \]
    for every log blowup $\tilde{X} \to X$ of a constant free log scheme $X$ by a locally constant free log scheme $\tilde{X}$. Then $I$ is precisely the kernel of \eqref{eqn:vogelpresmap}.
    
    Furthermore, $I$ is also generated by relations $[T]_{\lcf} - [\tilde{T}]_\lcf$ for a smooth toric blowup $\tilde{T} \to T$ of a smooth toric variety.
\end{proposition}

\begin{proof}
    By Proposition \ref{prop:niceblowup}, for the first part it suffices to show that for all triples of log modifications $Y \to X, \tilde{X} \to X, \tilde{Y} \to Y$ with $X,Y \in \logvar{k}, \tilde{X}, \tilde{Y} \in \logvar{k}^\lcf$ we have $\bra{\tilde{Y}}_{\lcf} - \bra{\tilde{X}}_\lcf \in I$.

    By refining $\tilde{Y}$, we can reduce to the case that there is a log blowup $\tilde{Y} \to \tilde{X}$. Now, stratify $\tilde{X} = \bigsqcup_i \tilde{X}_i$ into constant free log schemes, which induces by pullback a stratification $\tilde{Y} = \bigsqcup_i \tilde{Y}_i$. By definition of $I$, we have $\bra{\tilde{Y}_i}_\lcf - \bra{\tilde{X}_i}_\lcf \in I$, and hence $\bra{\tilde{Y}}_{\lcf} - \bra{\tilde{X}}_\lcf \in I$.

    Now let $\tilde X \to X$ be a log blowup of a constant free log scheme $X$ by a locally constant free log scheme $\tilde X$. There is a further log blowup $\tilde X' \to \tilde X$ fitting into a pullback diagram 
    \[
    \begin{tikzcd}
        \tilde X' \lpbstrict \ar[r] \ar[d]       &  T_0 \lpbstrict \ar[r] \ar[d]    &T \ar[d]      \\
        X \ar[r]       &\vec 0 \ar[r]         &\Aff^r
    \end{tikzcd}    
    \]
    where $T \to \Aff^r$ is a smooth toric blowup.
    The ideal $I$ is then generated by relations of the form $\bra{T_0}_\lcf = P^r$. This relation is itself the difference of the log blowup relation $[T]_{\lcf} - [\tilde{T}]_\lcf$ for $T = \Aff^r$ and $T = \Aff^r \setminus 0$.  
\end{proof}

Now it remains to compute classes of smooth toric varieties, as per the following proposition.

\begin{proposition}
\label{prop:toriclcf}
Let $X$ be a smooth toric variety of dimension $n$, with fan $\Sigma$. Then
\[
\bra{X}_{\lcf} \equiv \bra{\GG_m}_{\lcf}^n + (1-\chi_c(\Sigma)) P\bra{\GG_m}_\lcf^{n-1} \pmod {P(P+\bra{\GG_m}}),
\]
where $\chi_c$ is the compactly supported Euler characteristic.
\end{proposition}
\begin{proof}
Recall that the cones of $\Sigma$ correspond exactly to the toric strata of $X$. Here the cone $0$ corresponds to the torus $\GG_m^n$, and a cone $\sigma$ of dimension $r = r(\sigma)$ corresponds to a stratum $X_\sigma$ isomorphic to $\GG_m^{n-r}$ with log structure constant of rank $r$. All in all we find
\begin{align*}
\bra{X}_\lcf &= \bra{\GG_m}_\lcf^n + \sum_{\sigma \ne 0} \bra{X_\sigma}_\lcf &\\
             &= \bra{\GG_m}_\lcf^n + \sum_{\sigma \ne 0} [\GG_m]^{n-r(\sigma)} P^{r(\sigma)} &\\
             &\equiv \bra{\GG_m}_\lcf^n + \sum_{\sigma \ne 0} (-1)^{r-1} [\GG_m]^{n-1} P & \pmod {P(P+\bra{\GG_m})}\\
             &= \bra{\GG_m}_\lcf^n +  [\GG_m]^{n-1} P \sum_{\sigma \ne 0} (-1)^{r(\sigma)-1} &\\
             &= \bra{\GG_m}_{\lcf}^n + \chi_c(\Sigma \setminus 0)P\bra{\GG_m}^{n-1}&
\end{align*}
as required.
\end{proof}

\begin{proof}[{Proof of Theorem~\ref{thm:vogelpresentation}}]
By Corollary~\ref{cor:psisurjective}, we have a surjective map
\[
\Psi: \ordkvar{k}[P] \to \klvar{k}.
\]
By Proposition~\ref{prop:basicblowuprelns} the kernel $I$ of $\Psi$ is generated by $\bra{\tilde{X}}_\lcf - \bra{X}_\lcf$ for toric blowup $\tilde{X} \to X$ of smooth toric varieties. The fans of $\tilde{X}$ and $X$ have the same underlying topological space and hence the same Euler characteristic. By Proposition~\ref{prop:toriclcf}, we find $I \subseteq \langle P(P+\bra{\GG_m}) \rangle$. 

We have $P(P + \bra{\GG_m}) \in I$ by Example~\ref{ex:vogelrlns}, and we conclude
\[
\klvar{k} = \ordkvar{k}[P]/(P(P+\bra{\GG_m})).\qedhere
\]
\end{proof}

As a main example, we can compute the class of any toric variety.

\begin{proposition}
\label{prop:toricclass}
Let $X$ be a toric variety of dimension $n$, with fan $\Sigma$. Then
\[
\bra{X} = \bra{\GG_m}^n + (1-\chi_c(\Sigma)) P\bra{\GG_m}^{n-1}.
\]

\end{proposition}
\begin{proof}
The smooth case is Proposition~\ref{prop:toriclcf}. The general case follows as both sides are invariant under log blowups.
\end{proof}

\begin{example}
    Let $X$ be a proper toric variety of dimension $n$. Then the fan $\Sigma$ is a complete subdivision of $\RR^n$. The Euler characteristic $\chi_c$ is multiplicative and $\chi_c(\RR) = -1$, $\chi_c(\Sigma) = (-1)^n$.

    So \[\bra{X} = \bra{\GG_m}^n + (1-(-1)^n) P\bra{\GG_m}^{n-1}.\]

    For another way to compute the same formula, note that $X$ and $\PP^n$ have a common log blowup and hence the same class in $\klvar{k}$. The space $\PP^n$ has $\binom{n+1}{r+1}$ strata of dimension $r$, and using the relation $P(P+\bra{\GG_m}) = 0$ and the binomial expansion of $(1-1)^{n+1}$ one arrives at the same formula.

\end{example}

\begin{example}
    Let $X$ be a toric variety of dimension $n$ whose fan $\Sigma \subset \RR^n$ is obtained by taking the fan over a polyhedral complex $Q \subset \RR^{n} \setminus 0$. Then we have
   \[
   \bra{X} = \bra{\GG_m}^n + \chi_c(Q) P\bra{\GG_m}^{n-1}.
   \]

   If $X$ is an affine toric variety of dimension $n$. Then $\Sigma$ is a single cone of dimension $d \leq n$. Choose a sharp map from $f : \Sigma \to \NN$ using using \cite[Proposition I.2.2.1]{ogusloggeom}. Then $\Sigma$ is the fan over a polyhedral complex $f^{-1}(1)$ that is contractible if $d > 0$ and is empty otherwise. If $d > 0$, then
   \[\bra{X} = [\GG_m]^n + P \bra{\GG_m}^{n-1}.\]
\end{example}

The presentation from Theorem~\ref{thm:vogelpresentation} leads to two natural maps $\klvar{k} \to \ordkvar{k}$.

\begin{definition}
    Define two morphisms $\tau, \rho : \klvar{k} \to \ordkvar{k}$ of $\ordkvar{k}$-algebras by
    
    \[\tau(P) = 0\]
    and
    \[\rho(P) = -\GG_m\]
    respectively. They are the \emph{log Betti map} $\tau$ and the \emph{log Hodge map} $\rho$.
\end{definition}

\begin{example}
    For a toric variety $X$ of dimension $n$ we find $\tau(\bra{X}) = \bra{\GG_m^n}$ and $\rho(\bra{X})) = \chi_c(\Sigma) \bra{\GG_m^n}$.
\end{example}

\begin{remark}
The function sending $X$ to its underlying scheme $X^\circ$ does \emph{not} induce a morphism 
\[\klvar{k} \dashrightarrow \ordkvar{k}\]
because it does not satisfy the log blowup relations.
\end{remark}

A natural consequence of the presentation in Theorem \ref{thm:vogelpresentation} is that ring morphisms $f: \klvar{k} \to R$ are uniquely determined by $f|_{\ordkvar{k}}$ and $f(P)$.

\begin{proposition}\label{prop:logeulerchar}
There is a unique ring homomorphism 
\[\chi^{\rm log} : \klvar{\CC} \to \ZZ\]
extending the classical Euler characteristic \[\chi: \ordkvar{\CC} \to \ZZ . \]

We have $\chi^{\rm log} = \chi \circ \tau$, i.e. $\chi^{\rm log}(X) = \chi(X_0)$, with $X_0$ the locus where the log structure is trivial.

\end{proposition}
\begin{proof}
    By the presentation given in Theorem \ref{thm:vogelpresentation}, an extension $f$ of $\chi$ is determined by $f(P)$ and must satisfy $f(P)(f(P) + \chi(\GG_m)) = 0$. As $\chi(\GG_m) = 0$, we find $f(P) = 0$, and hence $f = \chi \circ \tau$.
\end{proof}

\begin{remark}
    Take $X \in \logvar{\CC}$. Its log Betti cohomology is defined as the singular cohomology with $\QQ$-coefficients
    \[H^*(\KN{X}, \QQ)\]
    of the Kato-Nakayama space $\KN{X}$, a real blowup of $X$ (written $X^{\rm log}$ in the original paper \cite{katonakayama}). Taking compact support and the alternating sum gives an alternative construction of $\chi^{\rm log}$ as 
    \[\chi^{\rm log}(X) = \sum (-1)^q \dim_\QQ H^q_c(\KN{X}, \QQ).\]
    
\end{remark}

\section{The log Hodge map}
\label{sec:hodge}

In this section, we work over $\Spec \CC$. Let $X$ be a log scheme over $\Spec \CC^\circ$ with smooth, projective underlying scheme $X^\circ$. 

Define
\[\Elog(X) \coloneqq \sum \dim H^q(\wedge^p \lkah{X})u^p v^q \qquad \in \ZZ[u, v].\]
If $X = X^\circ$ has trivial log structure, this coincides with the usual Hodge--Deligne $e-$polynomial $\Elog(X) = e(X)$.

\begin{example}
\label{ex:elogp1}
Let $X$ be $\PP^1$ with the toric log structure. Then $\lkah{X}$ is the trivial line bundle over $X$. Its cohomology vanishes, so we have $\Elog(X) = 1 + u$.
\end{example}

\begin{example}
\label{ex:elogp}
Let $X = P$ be the standard log point. Then $\lkah{X}$ is again the trivial line bundle over $P$ and $\Elog(P) = 1 + u$.
\end{example}

One would like to define $\Elog$ for non-projective varieties using $\Elog$ on projective ones by checking $\Elog$ satisfies the scissor relations. Unfortunately, it does not.

\begin{example}\label{ex:Epolynotwelldef}
We compactify $\GG_m$ in two different ways, embedding it in $\PP^1$ with its toric log structure and in $(\PP^1)^\circ$ with the trivial log structure.

Examples \ref{ex:elogp1} and \ref{ex:elogp} show 
\[
\Elog(\PP^1) = \Elog(P) = 1 + u.
\]

The $\Elog$-polynomials of their underlying schemes are the same as the usual $e$-polynomials:
\[\Elog((\PP^1)^\circ) = 1 + uv, \qquad \Elog(\pt) = 1.\]

Note the differences
\[\Elog(\PP^1) - 2 \Elog(P) = -1 -u\]
\[\Elog((\PP^1)^\circ) - 2 \Elog(\pt) = 1 + uv - 2 = uv - 1\]
do not coincide, and hence $\Elog$ does not respect the strict scissor relations. However, the differences coincide after setting $v=-1$.
\end{example}

\begin{proposition}
\label{prop:nologhodge}
    There is no map
    \[
    \phi: \logvar{\CC} \to \ZZ[u,v]
    \]
    that satisfies $\phi(X) = \phi(Z) + \phi(X \setminus Z)$ for strict closed subschemes $Z \subset X$ and that agrees with $\Elog$ on smooth log smooth projective schemes.

\end{proposition}
\begin{proof}
    Note that $\PP^1, {\PP^1}^\circ$ and $\pt$ are all smooth, log smooth and projective. As \[\Elog(\PP^1) - \Elog({\PP^1}^\circ) + 2 \Elog(\pt) = -2 + u - uv\] is not divisible by $2$, there is no possible value of $\phi(P) \in \ZZ[u,v]$ such that $\Elog(\PP^1) - 2 \phi(P) = \Elog({\PP^1}^\circ) - 2 \Elog(\pt)$.
\end{proof}

The $\Elog$-polynomial does not induce a function 
\[\klvar{\CC} \to \ZZ[u, v].\]
Nevertheless, it sometimes satisfies invariance under log modifications. 

\begin{lemma}\label{lem:pfwdE'polys}
    Let $\pi : X \to Y$ be a morphism of log schemes with smooth, projective underlying schemes. If $\pi$ is log \'etale and $R\pi_* \OO_X \simeq \OO_Y$, the $\Elog$-polynomials agree
    \[\Elog(X) = \Elog(Y).\]
\end{lemma}

\begin{proof}
Because $X \to Y$ is log étale, pullback equates the log Kähler differentials
\[\pi^*\lkah{Y} \simeq \lkah{X}.\]
The same goes for their exterior powers. Coherent sheaves such as $\lkah{X}, \lkah{Y}$ are perfect, as $X, Y$ have smooth underlying schemes. Then we can apply the projection formula \cite[0B54]{sta}:
\[R\pi_* \wedge^p \lkah{X} = R \pi_* \pi^* \wedge^p \lkah{Y} = \wedge^p \lkah{Y} \otimes R \pi_* \pi^* \OO_Y \cong \wedge^p \lkah{Y}.\]
The (cohomology of) wedge products of $\lkah{X}$ and $\lkah{Y}$ agree, so the $\Elog$ polynomials of $X$ and $Y$ agree.

\end{proof}

\begin{corollary}\label{cor:E'polyslogmodlogsmooth}
    Let $\pi : \tilde X \to X$ be a log modification with $X, \tilde{X}$ smooth, log smooth, projective. 
    
    Then the $\Elog$-polynomials agree \[\Elog(\tilde X) = \Elog(X).\]
\end{corollary}
\begin{proof}
    It suffices to show $R\pi_* \OO_{\tilde X} \simeq \OO_X$. This strict-\'etale local in $X$. After passing to a strict-\'etale neighborhood, one can find a pullback square
    \[
    \begin{tikzcd}
        \tilde X \ar[r] \ar[d] \lpbstrict        &\scr B \ar[d, "\tau"]         \\
        X \ar[r]       &\scr C
    \end{tikzcd}    
    \]
    with $X \to \scr C$ strict and $\scr B \to \scr C$ a representable, proper morphism of Artin fans.  

    The structure sheaf pushes forward along $\tau$
    \[R\tau_* \OO_{\scr B} \simeq \OO_{\scr C}\]
    as in \cite[Lemma 2.1]{kthylogprodfmla}. The map $X \to \scr C$ is smooth, as it is log smooth and strict. Cohomology and base change then ensures $R\pi_* \OO_{\tilde X} \simeq \OO_X$. 
\end{proof}

\begin{proposition}\label{prop:blowupEpolys}

Let $X = (X, \NN^k)$ be a free constant log scheme over $\Spec \CC^\circ$ and $\tilde X \to X$ a log blowup factoring through the blowup at $(e_1, e_2, \cdots, e_k) \subseteq \NN^k$. If $X^\circ, \tilde X^\circ$ are smooth and projective, their $\Elog$-polynomials coincide:
\[\Elog(\tilde X) = \Elog(X).\]    
\end{proposition}

\begin{proof}[{Proof of Proposition \ref{prop:blowupEpolys}}]

We show $R\pi_* \lkah{\tilde X} \overset{?}{=} \lkah{X}$. We can assume $X$ is connected and thus atomic. Let $Q$ be a log point with underlying scheme $\Spec \CC$ and a strict map $X \to Q$. The log blowup $\pi$ is pulled back from a log blowup $\tau$ of the log point
\[
\begin{tikzcd}
\tilde X \ar[d, "\pi", swap]  \ar[r] \lpbstrict       &\tilde Q \ar[d, "\tau"]      \\
X \arrow[r]        &Q.
\end{tikzcd}
\]
The pullback $\tilde X$ is the product of $X$ and $\tilde Q$. All the maps in the square are flat. 

If we show $R\pi_* \OO_{\tilde X} = \OO_X$, we can apply Lemma \ref{lem:pfwdE'polys}. Theorem \ref{thm:rothemail} in the appendix asserts
\begin{equation}\label{eqn:pfwdstrsheafpoint}
R\tau_* \OO_{\tilde Q} = \OO_Q.   
\end{equation}
By cohomology and base change, we are done.

\end{proof}

\begin{definition}
    
Write $\Elogbar(X)$ for the image of the polynomial $\Elog(X)$ under the ring homomorphism 
\begin{align*}
    \ZZ[u, v] &\to \ZZ[u] \times \ZZ[v]\\
            v &\mapsto (-1,v)\\
            u &\mapsto (u,0)\\
\end{align*}

Define 
\begin{align*}
\Elogbar_1(X) \coloneqq \Elog(X)(u,-1)\\ 
\Elogbar_2(X) \coloneqq \Elog(X)(0,v),
\end{align*}
or equivalently
\[\Elogbar(X) = (\Elogbar_1(X), \Elogbar_2(X)) \coloneqq \left(
\sum \chi(\wedge^p \lkah{X}) u^p, \quad 
\sum \dim H^q(\OO_X) v^q
\right) \qquad 
\in \ZZ[u] \times \ZZ[v]\] 
as the generating functions of the Euler characteristics of the wedges $\wedge^p \lkah{X}$ and the Betti numbers. 

\end{definition}

We construct maps out of $\klvar{k}$ which restrict to $\Elogbar_1, \Elogbar_2$ on nice log schemes with smooth, projective underlying scheme.

\subsection{The cohomology of the structure sheaf}\label{ss:cohomstrsheaf}
For a log scheme $X$ with smooth, projective underlying scheme, the polynomial $\Elogbar_2(X)$ coincides with the substitution $u = 0$ of the usual $e-$polynomial
\[\Elogbar_2(X) = e(X^\circ)(0, v).\]
Write $b$ for the composite 
\[b : \ordkvar{\CC}[P] \overset{P \mapsto \Spec \CC}{\longrightarrow} \ordkvar{\CC} \overset{e|_{u = 0}}{\longrightarrow} \ZZ[v].\]

\begin{proposition}
    The map $b$ factors through the quotient $\ordkvar{\CC}[P] \to \klvar{\CC}$.  
\end{proposition}

\begin{proof}
    We need only check $P(P + [\GG_m])$ maps to zero. In fact, $b$ sends $P + [\GG_m]$ to zero, as $e|_{u = 0}(\Spec \CC) = 1$ and $e|_{u = 0}(\GG_m) = (uv-1)|_{u = 0} = -1$.
\end{proof}

This defines a log motivic invariant which restricts to $\Elogbar_2$ on log schemes with smooth, projective underlying scheme.

\subsection{The log $\chi_y$-genus}

    The alternating sum of the Euler characteristics of the usual K\"ahler differentials is called the $\chi_y$-genus
    \[\chi_y(X) \coloneqq \sum (-1)^p \chi(\wedge^p \kah{X}) y^p.\]
    The $\Elogbar_1$-polynomial defined here is a ``log $\chi_{-y}$-genus,'' where we replace $y$ by $-y$. 

    If $X = (X, D)$ is an s.n.c.~pair of dimension $\dim X = n$ with $X$ proper, \cite[Proposition 3.1]{grossrefinedtropicalization} relates our $\Elogbar_1$-polynomial to the $\chi_y$-genus of $X \setminus D$, after a change of coordinates:
    \begin{align}
    \chi_y(X \setminus D) &= \sum (-1)^{n-r} \chi(\wedge^{n-r} \lkah{X}) y^r,    \\
    (-u)^n \cdot \chi_{-\frac{1}{u}}(X \setminus D) &=\Elogbar_1(X) \coloneqq \sum \chi(\wedge^p \lkah{X}) u^p \label{eq:chiyelogbar1}.
    \end{align}

\begin{remark}
    This equality shows $\Elogbar_1(X)$ only depends on the interior $X \setminus D$ where the log structure is trivial for s.n.c.~pairs. In particular, this implies invariance under log blowups of s.n.c.~pairs. It does not imply $\Elogbar_1(P) = 0$, as we still have $\Elogbar_1(P) = 1 + u$. 
\end{remark}

We will now build a ring homomorphism out of the log Grothendieck ring and show it agrees with $\Elogbar_1$ for a large collection of log schemes. 

\begin{definition}
    Let $\t$ be the composite
    \[\klvar{\CC} \overset{\rho}{\longrightarrow} \ordkvar{\CC} \overset{e}{\longrightarrow} \ZZ[u, v]\]
    of the log Hodge map $\rho$ (Theorem \ref{thm:vogelpresentation}) with the usual $e$-polynomial. Write $\tbar$ for the composite of $\t$ with the map
    \[\ZZ[u, v] \longrightarrow \dfrac{\ZZ[u, v]}{uv + u} \to \ZZ[u] \times \ZZ[v]\]
    
    and $\tbar = (\tbar_1, \tbar_2)$. 
\end{definition}

\begin{example}
If $(X, D)$ is an s.n.c.~divisor with components $D_i \subseteq D$, let $D_r$ be the union of the closed strata $\bigcap_I D_i$ with $\# I = r$. For example, 
\[D_0 = X, \quad D_1 = D, \quad D_2 = \bigcup_{i \neq j} (D_i \cap D_j), \cdots.\]
The log Hodge map sends the divisorial log structure $X = (X, D)$ to the alternating sum
\[\rho(X) = \sum_{r \in \NN} (-[\GG_m])^r [D_r^\circ \setminus D^\circ_{r+1}] = [X^\circ \setminus D^\circ] - [\GG_m] \cdot [D^\circ \setminus D^\circ_2] + \cdots.\]

The map $\t$ sends $P$ to the $e$-polynomial of $-[\GG_m]$
\[\t(P) = 1-uv.\]
So $\t$ is computed by
\[
\t(X) = \sum_k (1-uv)^r e(D_k^\circ \setminus D_{r+1}^\circ).
\]

\end{example}

We can now identify $\Elogbar$ and $\tbar$ in two nice cases. 

\begin{theorem}\label{thm:E=t}
    The two maps $\Elogbar, \tbar$ agree on constant, free log schemes $(X, \NN^r)$ with smooth, projective underlying scheme $X$. 
\end{theorem}

\begin{proof}
There is a pullback square
\begin{equation}\label{eqn:strpbsquareloghodge}
\begin{tikzcd}
X \ar[r] \ar[d] \lpbstrict       &P_r \ar[d]        \\
X^\circ \ar[r]         &\pt,
\end{tikzcd}
\end{equation}
where $P_r$ is the point with rank-$r$ log structure. The resulting equality of classes $[X] = [X^\circ] P^r$ identifies the $\tbar$-polynomial
\[\tbar(X) = \tbar([X^\circ]) \tbar(P)^r.\]
The $\tbar$-polynomial of $P$ is defined as
\[\tbar(P) \coloneqq e(-[\GG_m]) = 1-uv \mod u + uv\] and hence is equal to $\Elogbar(P) = 1 + u \mod u + uv$.

By \cite[Proposition IV.1.2.15]{ogusloggeom}, the pullback square \eqref{eqn:strpbsquareloghodge} decomposes the log K\"ahler differentials of $X$ as
\[\lkah{X} = \lkah{P_r}|_X \oplus \kah{X^\circ}.\]
The log K\"ahler differentials of the point $P_r$ are the trivial bundle $\lkah{P_r} = \OO_{P_r}^{\oplus r}$ of rank $r$. 

The exterior algebra is then
\[
\wedge^* \lkah{X} = \wedge^* \OO^r_X \otimes \wedge^* \kah{X^\circ}.
\]
We have $\wedge^i \OO_X^r = \OO_X^{\binom{r}{i}}$ for $0 \leq i \leq r$, and hence
\[
\wedge^p \lkah{X} = \bigoplus_{i + j = p} \left(\wedge^j \kah{X^\circ}\right)^{\oplus \binom{r}{i}}.
\]

The log Hodge numbers are
\[
\dim H^q(\wedge^p \lkah{X}) = 
\sum_{i + j = p} \dbinom{r}{i} h^{j, q}(X^\circ),
\]
where $h^{p, q}$ are the usual Hodge numbers of the underlying scheme $X^\circ$. The generating function $\Elog$ for the log Hodge numbers then equals
\[
\Elog(X) = 
\sum_{p, q} \left(\sum_{i + j = p} \dbinom{r}{i} h^{j, q}(X^\circ)\right) u^p v^q = e(X^\circ)(1+u)^r.
\]

We obtain
\begin{align*}
    \Elogbar(X) &= e(X^\circ)(1+u)^r \mod u + uv       \\
                &= \tbar([X^\circ])\tbar(P)^r      \\
                &= \tbar(X).       
\end{align*}
\end{proof}

\begin{theorem}\label{thm:E=tlogsmoothandreasgross}
    The equality $\Elogbar(X) = \tbar(X)$ holds when $X$ is log smooth and has smooth, projective underlying scheme. 
\end{theorem}

\begin{proof}
    If $\tilde X \to X$ is a log modification and both $\tilde X, X$ have smooth, projective underlying scheme, Corollary \ref{cor:E'polyslogmodlogsmooth} equates their $\Elog$-polynomials. The $\t$-polynomials are similarly unchanged. Section \S \ref{ss:cohomstrsheaf} shows $\Elogbar_2(X) = \tbar_2(X)$. 

    To show that $\Elogbar_1(X)=\tbar_1(X)$, we may, after applying a suitable log blowup, assume that $X = (X, D)$ is an s.n.c.~pair. We will now prove the assertion with induction first on the dimension of $X$ and then on the number of irreducible components of the divisor $D$. For the base cases we note that case $\dim X = 0$ is clear and that if $D = \varnothing$, both $\Elogbar_1(X)$ and $\tbar_1(X)$ are immediately equal to the classical $e$-polynomial of $X = X^\circ$, evaluated at $v = -1$.

Now assume $D\neq \varnothing$, let $F$ be a component of $D$, and let $D'=D-F$. Then $F$ meets $D'$ transversely and $D' \cup F = D$. Moreover, let $X'$ be the scheme $X^\circ$ endowed with the divisorial log structure defined by $D'$ and denote by $\hat F$ the scheme $F$ endowed with the divisorial log structure defined by the s.n.c.~divisor $D'\vert_F$.
By \cite[Property 2.3 b)]{EsnaultViehweg} we have, for every $p\in \ZZ$, a short exact sequence
\[
0 \to 
    \wedge^p\lkah{X'} \to 
        \wedge^p\lkah{X} \to 
            \wedge^{p-1} \lkah{\hat F} \to 
                0 \ .
\]

By the additivity of the holomorphic Euler characteristic it follows that 
\[
    \Elogbar_1(X) = \Elogbar_1(X') + u \Elogbar_1(\hat F),
\]
which, by induction hypothesis, equals
\[
\tbar_1(X') + u \tbar_1(\hat F)
\]

    It remains to show that $\tbar_1(X') + u \tbar_1(\hat F)$ is equal to $\tbar_1(X)$. 
    We first prove that $[F] = [\hat F] \cdot P$ in $\klvar{\CC}$.
    
    The class $[F]$ is the sum $[F \setminus D'] + [D']$. The log scheme $[F \setminus D']$ has locally constant, rank-one log structure. It's equivalent to the constant rank-one log structure by Corollary~\ref{cor:locconstequivconst}
    \[[F \setminus D'] = [\hat F \setminus D'] \cdot P.\]
    Apply the same argument to the strata of the intersections of $D'$. 

    We can conclude that $\tbar_1(F) = \tbar_1(\hat F) \tbar_1(P) = \tbar_1(\hat F)(1 + u)$, so $u\tbar_1(\hat F) = \tbar_1(F) - \tbar_1(\hat F)$. But the classes
    \[[X'] - [\hat F] = [X] - [F]\]
    are equal. So their $t$-polynomials agree and we obtain
    \[\tbar_1(X) = \tbar_1(X') + \tbar_1(F) - \tbar_1(\hat F) = \tbar_1(X') + u \cdot \tbar_1(\hat F)\]
    and we are done. 
\end{proof}

\subsection{Duality}

Write $\LL \coloneqq [(\Aff^1)^\circ]$
for the Lefschetz motive, the class of the affine line with trivial log structure. 

After inverting $\LL$, the usual Grothendieck ring of varieties $\ordkvar{\CC}$ admits a ring involution
\[ (-)^\vee \colon \ordkvar{\CC}[\LL^{-1}] \to \ordkvar{\CC}[\LL^{-1}] \]
determined by $[X] \mapsto [X]^\vee = [X] / \LL^{\dim X}$ for all smooth projective complex varieties $X$ \cite{Bittner2004}. 

\begin{example}
    The dual of $[\PP^1]$ is 
    \[[\PP^1]^\vee = [\PP^1] / \LL = (\LL + 1) / \LL = 1 + \LL^{-1}.\]
    As $(-)^\vee$ is a ring homomorphism, 
    \[\LL^\vee = \LL^{-1}, \qquad [\GG_m]^\vee = \LL^{-1} - 1 = -\dfrac{\bra{\GG_m}}{\LL}.\] 
\end{example}

We show duality $(-)^\vee$ extends in two ways to the log Grothendieck ring of varieties
\[
i_1, i_2 : \klvar{\CC}[\LL^{-1}] \longrightarrow \klvar{\CC}[\LL^{-1}].
\]
Define 
\begin{equation}\label{eqn:dualities}
    i_1(P) \coloneqq -P \cdot \LL^{-1}, \qquad i_2(P) \coloneqq (P + [\GG_m]) \cdot \LL^{-1}
\end{equation}
and let $i_1([X^\circ]) = i_2([X^\circ]) \coloneqq [X^\circ]^\vee$ extend the usual duality on schemes $X^\circ$ with trivial log structure. 

\begin{lemma}
    The assignments from \eqref{eqn:dualities} yield well-defined ring homomorphisms which are involutions $i_1^2 = i_2^2 = id$. They satisfy
    \begin{equation}\label{eqn:dualitycompatible1}
    (-)^\vee \circ \tau = \tau \circ i_1, \qquad (-)^\vee \circ \rho = \rho \circ i_1
    \end{equation}
    \begin{equation}\label{eqn:dualitycompatible2}
    (-)^\vee \circ \tau = \rho \circ i_2, \qquad (-)^\vee \circ \rho = \tau \circ i_2
    \end{equation}
    with the log Betti and log Hodge maps $\tau, \rho$. 
\end{lemma}

\begin{proof}
    To check $i_1, i_2$ are well-defined, we need to show they kill $P(P + [\GG_m])$:
    \begin{align*}
    i_1(P \cdot (P + [\GG_m])) &= \dfrac{-P}{\LL} \left( \dfrac{-P}{\LL} - \dfrac{[\GG_m]}{\LL}\right)    \\
        &=P \cdot (P + [\GG_m]) \cdot \LL^{-2} = 0.         \\
    i_2(P \cdot (P + [\GG_m])) &= \dfrac{P + [\GG_m]}{\LL} \cdot \left(
    \dfrac{P + [\GG_m]}{\LL} - \dfrac{[\GG_m]}{\LL}
    \right)         \\
        &=(P + [\GG_m]) \cdot P \cdot \LL^{-2} = 0.
    \end{align*}
    They are involutions as we have $i_1^2(P) = i_2^2(P) = P$. 
\end{proof}

The ring involutions $i_1, i_2$ are determined by duality $(-)^\vee$ on schemes with trivial log structure and \eqref{eqn:dualitycompatible1}, \eqref{eqn:dualitycompatible2}, up to a small discrepancy. 

\begin{lemma}
    Let $F : \klvar{\CC}[\LL^{-1}] \to \klvar{\CC}[\LL^{-1}]$ be a ring endomorphism extending $(-)^\vee$. If $F$ satisfies \eqref{eqn:dualitycompatible1}, the difference $F - i_1$ is annihilated by $[\GG_m]$
    \[[\GG_m] \cdot (F - i_1) = 0.\]
    The same goes for $\eqref{eqn:dualitycompatible2}$ and $i_2$. 
\end{lemma}

\begin{proof}
    Let $F : \klvar{\CC}[\LL^{-1}] \to \klvar{\CC}[\LL^{-1}]$ be a ring endomorphism extending $(-)^\vee$ and write $F(P) = \alpha + \beta P$ for some $\alpha, \beta \in \ordkvar{\CC}[\LL^{-1}]$.

    If $F$ satisfies \eqref{eqn:dualitycompatible1}, check $\alpha = 0$ and 
    \[
    [\GG_m] \cdot (\beta + \LL^{-1}) = 0.
    \]
    If $F$ satisfies \eqref{eqn:dualitycompatible2}, $\alpha = \dfrac{\GG_m}{\LL}$ and 
    \[
    [\GG_m] \cdot (\beta - \LL^{-1}) = 0.
    \]
\end{proof}

If we invert $[\GG_m]$ by tensoring with $\ordkvar{\CC}[[\GG_m]^{-1}]$, then $i_1, i_2$ are the unique ring endomorphisms extending $(-)^\vee$.

\begin{example}
    Let $X$ be a proper toric variety of dimension $n$. Then using Proposition~\ref{prop:toricclass} we can compute
    \[i_j(X) = (-1)^{n(j+1)} \LL^{-n} \cdot [X].\]
\end{example}

\begin{remark}[Duality]
    
For a smooth, projective scheme $X$, Serre duality equates the Hodge numbers $h^{p, q} = h^{n-p, n-q}$. The $e$-polynomial of the dual $X^\vee$ is then
\[e(X^\vee) = e(X)/e(\LL^{\dim X}) = e(X)/(uv)^{\dim X} = e(X)(u^{-1}, v^{-1})\]
and the square
\[
\begin{tikzcd}
    \ordkvar{\CC}[\LL^{-1}] \ar[r, "e"] \ar[d, "-^\vee", swap]       &\ZZ[u^\pm, v^\pm] \ar[d, "I"]      \\
    \ordkvar{\CC}[\LL^{-1}] \ar[r, "e", swap]      &\ZZ[u^\pm, v^\pm]
\end{tikzcd}
\]
commutes, where $I(p) \coloneqq p(u^{-1}, v^{-1})$. 

The equality $\rho \circ i_1 = (-)^\vee \circ \rho$ gives 
\begin{equation}\label{eqn:tlogserreduality}
    \tbar_1(i_1(X)) = \tbar_1(X)(u^{-1}).
\end{equation}
Endow a smooth, projective scheme $X^\circ$ of dimension $n$ with constant free rank-$k$ log structure $X$. Then $i_1(X) = (-1)^k \LL^{-k} P^k {X^\circ}^{\vee} = (-1)^k \LL^{-(k+n)} X$ and Theorem \ref{thm:E=t} rewrites \eqref{eqn:tlogserreduality} as
\[\dfrac{(-1)^k \tbar_1(X)}{\tbar_1(\LL^{k+n})} = \dfrac{(-1)^k \Elogbar_1(X)}{u^{k+n}} = \Elogbar_1(X)(u^{-1}).\]
Whence $\chi(\wedge^{k+n-i} \lkah{X}) = (-1)^k \chi(\wedge^i \lkah{X})$, a limited ``log Serre duality.''

\end{remark}

\begin{example}
    Let $X = (\PP^1, \NN)$ be the projective line with constant rank-one log structure. Its log K\"ahler differentials are
    \[\lkah{X} = \OO(-2) \oplus \OO,\]
    with wedges
    \[
    \wedge^p \lkah{X} = 
    \left\{
    \begin{tikzcd}[row sep = tiny]
        \OO      &p = 0   \\
        \OO(-2) \oplus \OO  &p = 1      \\
        \OO(-2)     &p = 2         \\
        0   &\text{otherwise}.
    \end{tikzcd}
    \right.
    \]
    Its log Hodge ``diamond'' is then the rectangle
    \[\begin{tabular}{c|c c c}
        $p$:      &0      &1        &2      \\ 
        \hline
         $q = 0$  &1  &1  &0                 \\
         $q = 1$  &0  &1  &1 
    \end{tabular},\]
    and 
    \[\chi(\OO) = 1, \qquad \chi(\lkah{X}) = 0, \qquad \chi(\wedge^2 \lkah{X}) = -1.\]
\end{example}

\begin{remark}\label{rmk:bittner}
    We sketch a ``Bittner presentation'' \cite{Bittner2004} for the log Grothendieck ring, which doesn't seem to work as well as for the ordinary Grothendieck ring. 
    
    Consider the free group generated by isomorphism classes of $\CC$-log schemes $X$ with smooth, projective underlying scheme. Let $\Bitt$ be the quotient of this free group by three relations: 
    \begin{itemize}
        \item $[\varnothing] \sim 0$,
        \item If $\tilde X \to X$ is a log blowup, then $[X] \sim [\tilde X]$,
        \item If $Z \subseteq X$ is a strict closed subscheme and $B \to X$ its blowup, write $E \subseteq B$ for the exceptional divisor. Endow $Z, B, E$ with log structure pulled back from $X$. Then 
        \[[B] - [E] \sim [X] - [Z].\]
    \end{itemize}

    Let $\Bitt' \subseteq \Bitt$ be the subgroup generated by constant log schemes with smooth, projective underlying scheme. The arguments of \cite{Bittner2004} identify $\Bitt' \simeq \klvar{\CC}$. It is not clear to us if this is the whole group $\Bitt' \overset{?}{=} \Bitt$.

    Given a finite type, separated variety $V$ with log structure and a compactification $V \subseteq \bar V$, it's unclear how to extend the log structure to $\bar V$. Moreover, general log schemes $X$ do not admit a ``resolution'' $\tilde X \to X$ by log smooth schemes in the usual sense, even over $\CC$. For the standard log point $P$, there is no map $X \dashrightarrow P$ from a nonempty log flat log scheme $X$. See \cite{minghaologresolution} for the case where $X$ is generically log smooth and nice. 
    
\end{remark}

\appendix

\section{Cohomology of fibers of toric blowups}\label{appendix:rothemail}

We work over a field $k$. Let $B \to \Aff^n$ a subdivision. Write $\vec 0 \in \Aff^n$ for the origin and $B_0 \subseteq B$ for the fiber over $\vec 0$. Suppose $B \to \Aff^n$ factors through the log blowup $Bl_{\vec 0} \Aff^n$ at the ideal $(x_1, \cdots, x_n) \subseteq k[x_1, \cdots, x_n]$.

The goal of this appendix is to compute the cohomology of $B_0$. The proof is due to Mike Roth, who graciously encouraged the authors to write up his arguments.

\begin{theorem}\label{thm:rothemail}
The cohomology of $B_0$ is $k$:
\[R\Gamma(B_0, \OO_{B_0}) = k.\]
In other words, $H^q(B_0, \OO_{B_0}) = 0$ for $q \neq 0$ and $H^0(B_0, \OO_{B_0}) = k$.
\end{theorem}

The main obstacle to proving this theorem is that $B_0$ need not be reduced. 

\begin{example}[Karl Schwede]
Let $B \to \Aff^2$ be the log blowup at the monoidal ideal
\[I = (x^2, y) \cdot (x, y) \cdot (x, y^2) \qquad \subseteq \NN^2.\]
The fiber over $\vec 0 \in \Aff^2$ has three components, and the middle one is nonreduced.

\end{example}

If we replace $B_0$ by its reduced subscheme $B_{0, red}$, a similar vanishing 
\[
H^*(B_{0, red}, \OO_{B_{0, red}}) = \CC
\]
was shown in \cite[Lemma 2.5]{mirkoshinder}. The proof refers to special cases in \cite[Proposition 3.1]{dubois}, \cite[Proposition 3.7 and \S 3.6]{mixedhodgesteenbrink}, \cite[Proposition 8.1.11.(ii) and 8.1.12]{ishiiintrosingularities}, \cite[Lemma 1.2]{namikawadeformation}. Theorem \ref{thm:rothemail} also generalizes a result of Molcho--Wise \cite[Corollary 6.4]{molchowiselogetaledescentnote}.

\begin{lemma}\label{lem:derpfwdstrsheafdivisor}
    Let $\pi : X \to Y$ be a birational, perfect morphism such that $R\pi_* \OO_X = \OO_Y$. Let $E \subseteq Y$ be an effective Cartier divisor and $D \subseteq X$ its pullback. Then the structure sheaf of the divisor $D$ pushes forward to that of $E$
    \[R\pi_* \OO_D = \OO_E.\]
\end{lemma}

\begin{proof}
    By the projection formula \cite[0B54]{sta} and our hypothesis, any line bundle $L$ on $Y$ satisfies
    \begin{equation}\label{eqn:derivedpfwdlinebund}
        R\pi_* \pi^* L = L.
    \end{equation}

    We claim the ideal sheaves pull back $\pi^* \OO_Y(-E) = \OO_X(-D)$. This is a local question, so assume $Y = \Spec B$ and $E = V(f)$. Then $\OO_Y(-E) \to \OO_Y$ is multiplication $f : B \to B$ by the nonzerodivisor $f$. The ideal sheaf $\OO_X(-D)$ is the image of the pullback $\pi^* \OO_Y(-E) \to \pi^* \OO_Y = \OO_X$, which we need to show is injective. Localize in $X$ to assume $X = \Spec A$. Then multiplication $f : A \to A$ is also injective, otherwise $X$ has an embedded component over $Y$. But the map $\pi$ is assumed to be birational.

    Consider the exact sequence 
    \[0 \to \OO_X(-D) \to \OO_X \to \OO_D \to 0\]
    on $X$. Apply $R\pi_*$ to get an exact triangle 
    $R\pi_* \OO_X(-D) \to R\pi_* \OO_X \to R \pi_* \OO_D \overset{+1}{\to}$.
    By taking $L = \OO_Y(E), \OO_Y$ in \eqref{eqn:derivedpfwdlinebund}, the triangle becomes 
    \[\OO_X(-E) \to \OO_X \to R \pi_* \OO_D \overset{+1}{\to}\]
    and we can see $R\pi_* \OO_D = \OO_E$. 
\end{proof}

\begin{proof}[{Proof of Theorem \ref{thm:rothemail}}]
    Write $\pi : B' \coloneqq Bl_{\vec 0} \Aff^n \to \Aff^n$ for the log blowup at the ideal $(x_1, \cdots, x_n)$. By assumption on $B$, it factors through this blowup $\rho : B \to B'$. Write $B_0, B'_0$ for the fiber over the origin $\vec 0 \in \Aff^n$. Apply Lemma \ref{lem:derpfwdstrsheafdivisor} to $\rho$ with $D = B_0, E = B'_0$. Then $R \rho'_* \OO_{B''_0} = \OO_{B'_0}$. But $B'_0 = \PP^{n-1}$, so $R\pi'_* \OO_{B'_0} = \OO_{\vec 0}$. Composing these derived functors, we get $R(\pi' \circ \rho')_* \OO_{B''_0} = \OO_{\vec 0}$. Taking global sections gives the result. 
\end{proof}

\bibliographystyle{alpha}
\bibliography{zbib}

\end{document}